\documentclass[12pt,reqno]{amsart}
\usepackage{hyperref}

\makeatletter
\def\@tocline#1#2#3#4#5#6#7{\relax
  \ifnum #1>\c@tocdepth % then omit
  \else
    \par \addpenalty\@secpenalty\addvspace{#2}%
    \begingroup \hyphenpenalty\@M
    \@ifempty{#4}{%
      \@tempdima\csname r@tocindent\number#1\endcsname\relax
    }{%
      \@tempdima#4\relax
    }%
    \parindent\z@ \leftskip#3\relax \advance\leftskip\@tempdima\relax
    \rightskip\@pnumwidth plus4em \parfillskip-\@pnumwidth
    #5\leavevmode\hskip-\@tempdima
      \ifcase #1
       \or\or \hskip 1em \or \hskip 2em \else \hskip 3em \fi%
      #6\nobreak\relax
    \dotfill\hbox to\@pnumwidth{\@tocpagenum{#7}}\par
    \nobreak
    \endgroup
  \fi}
\makeatother

\usepackage{tikz,amsthm,amsmath,amstext,amssymb,amscd,epsfig,euscript, mathrsfs, dsfont,pspicture,multicol,graphpap,graphics,graphicx,times,enumerate,subfig,sidecap,wrapfig,color}
\usepackage{ hyperref}
\usepackage{enumerate}

 \numberwithin{equation}{section}

\def\bR{{\mathbb{R}}}

\def\bZ{{\mathbb{Z}}}

\def\bN{{\mathbb{N}}}

\def\D{{\Delta}}
\def\G{{\Gamma}}

\def\cB{{\mathcal{B}}}
\def\cC{{\mathscr{C}}}
\def\cD{{\mathcal{D}}}

\def\cH{{\mathcal{H}}}

\def\cM{{\mathcal{M}}}

\def\cP{{\mathscr{P}}}

\def\cW{{\mathcal{W}}}

\def\bbeta{b\beta}

\def\one{\mathds{1}}

\def\ve{\varepsilon}

\renewcommand{\d}{{\partial}}
\def\lec{\lesssim}
\def\gec{\gtrsim}

\DeclareMathOperator{\diam}{diam}
\def\dist{\mathop\mathrm{dist}} 						%distance
\def\supp{\mathop\mathrm{supp}}					%support
\newcommand{\ps}[1]{\left( #1 \right)}

\newcommand{\cnj}[1]{\overline{#1}}

\def\Claim{ {\bf Claim: }}
\def\dq{\Delta_Q}

\def\XXint#1#2#3{{\setbox0=\hbox{$#1{#2#3}{\int}$ }
\vcenter{\hbox{$#2#3$ }}\kern-.58\wd0}}

\theoremstyle{plain}
\newtheorem{theorem}{Theorem}

\newtheorem{lemma}[theorem]{Lemma}

\newtheorem{proposition}[theorem]{Proposition}

\theoremstyle{definition}

\newtheorem{definition}[theorem]{Definition}
\newtheorem{remark}[theorem]{Remark}

\numberwithin{equation}{section}
\numberwithin{theorem}{section}

\def\Claim{ {\bf Claim: }}

  \DeclareFontFamily{U}{mathb}{\hyphenchar\font45} 
\DeclareFontShape{U}{mathb}{m}{n}{
      <5> <6> <7> <8> <9> <10> gen * mathb
      <10.95> mathb10 <12> <14.4> <17.28> <20.74> <24.88> mathb12
      }{}
\DeclareSymbolFont{mathb}{U}{mathb}{m}{n}

\DeclareMathSymbol{\toitself}      {3}{mathb}{"FD}  %*

\begin{document}

\title[Uniform domains with rectifiable boundaries]{Uniform domains with rectifiable boundaries and harmonic measure}

\author{Mihalis Mourgoglou}
\address{Departament de Matem\`atiques\\ Universitat Aut\`onoma de Barcelona and Centre de Reserca Matem\` atica\\ Edifici C Facultat de Ci\`encies\\
08193 Bellaterra (Barcelona) }
\email{mmourgoglou@crm.cat}
%\thanks{Supported in part by the grants RTG DMS 08-38212  and DMS-0856687}
%\thanks{The author was supported by grants ERC grant 320501 of the European Research Council (FP7/2007-2013).}

\maketitle

\begin{center}
\textsl{In memory of G. I. Chatzopoulos}
\end{center}

\begin{abstract}
We assume that $\Omega \subset \bR^{d+1}$, $d \geq 2$, is a uniform domain with lower Ahlfors-David $d$-regular and $d$-rectifiable boundary. We show that if $\cH^d|_{\partial \Omega}$ is locally finite, then the Hausdorff measure $\cH^d$ is absolutely continuous with respect to the harmonic measure $\omega$ on $\d \Omega$, apart from a set of  $\cH^d$-measure zero.
\end{abstract}
\tableofcontents

\section{Introduction}\label{sec:intro}

Determining (mutual) absolute continuity of the harmonic measure associated to the Laplace operator and the $d$-Hausdorff measure in domains with ``rough" boundaries has been a hot topic of research in mathematical analysis for almost four decades now. The interest in such questions can be justified partially by the connection between (a quantitative version of) the absolute continuity of the harmonic measure and the well-posedness of the Dirichlet problem with data in some $L^p$ space (even for elliptic operators of divergence form with merely bounded real coefficients). 
%There is a strong connection between the rectifiability of the boundary of a domain in Euclidean space and the possible absolute continuity of harmonic measure  with respect to Hausdorff measure. Recall that a set $E$ is $n$-rectifiable if it can be covered by a countable union of $n$-dimensional Lipschitz graphs up to a set of zero $n$-dimensional Hausdorff measure $\cH^{n}$. 

Already in 1916, F. and M. Riesz \cite{RR} showed that for simply connected planar domains, bounded by a Jordan curve, whose boundary has finite length, harmonic measure and arc-length are mutually absolutely continuous. Their theorem was improved by Lavrentiev \cite{Lav} demonstrating that in a simply connected domain in the complex plane, bounded by a chord-arc curve, the harmonic measure is in the $A_\infty$ class of Muckehoupt weights. Bishop and Jones \cite{BJ} proved a local version of F. and M. Riesz theorem by showing that if $\Omega$ is a simply connected planar domain and $\Gamma$ is a curve of finite length, then $\omega\ll \cH^{1}$ on $\d\Omega\cap \Gamma$, where $\omega$ stands for the harmonic measure. They also give an example of a domain $\Omega$ whose boundary is contained in a curve of finite length, but $\cH^{1}(\d\Omega)=0<\omega(\d\Omega)$, thus showing that some sort of connectedness in the boundary is required. 

In higher dimensions, the situation is a lot more delicate. The obvious generalization to higher dimensions is false due to examples of Wu and Ziemer: they construct topological two-spheres in $\mathbb{R}^{3}$ with boundaries of finite Hausdorff measure $\cH^{2}$ where either harmonic measure is not absolutely continuous with respect to $\cH^{2}$ \cite{Wu} or $\cH^{2}$ is not absolutely continuous with respect to harmonic measure \cite{Z}, respectively. In the affirmative direction, Dahlberg shows in \cite{Da} that in a Lipschitz domain, the harmonic measure and the $d$-Hausdorff measure restricted to the boundary are $A_\infty$-equivalent. The same result was proved by  David and Jerison in \cite{DJ90}  under the assumptions that $\Omega\subset \mathbb{R}^{d+1}$ is an NTA domain and $\d\Omega$ is Ahlfors-David regular. Recently, Azzam, Hofmann, Martell, Nystr\"om and Toro \cite{AHMNT} showed that any uniform domain with uniformly rectifiable boundary is an NTA domain and thus, $\omega \in A_\infty$ by \cite{DJ90} (a direct proof of the $A_\infty$-equivalence between $\omega$ and $\cH^d|_{\d \Omega}$ in this case was given earlier by Hofmann and Martell \cite{HM12}; the converse implication is proved in \cite{HMUT} and a stronger version of it in \cite{HM2}). One can also find similar results for domains with uniformly rectifiable boundaries (without the uniformity assumption) in \cite{BH}. Hofmann, Martell and Toro \cite{HMT} recently obtained a characterization of uniform domains with uniformly rectifiable boundaries via the $A_\infty$ equivalence of the elliptic harmonic measure and the $d$-Hausdorff measure (for second order elliptic operators of divergence form with real, locally Lipschitz coefficients that satisfy a natural Carleson condition). 

At first look, Ahlfors-David regularity seems superfluous for establishing absolute continuity in NTA domains, and in some sense it is: in \cite{Badger12}, Badger shows that if one merely assumes $\cH^{d}|_{\d\Omega}$ is locally finite and $\Omega\subset \mathbb{R}^{d+1}$ is NTA, then we still have $\cH^{d}|_{\d\Omega} \ll \omega$. He also shows that  $\omega \ll \cH^{d}|_{\d\Omega}\ll \omega$ on the set
\[\{x\in \d\Omega: \liminf_{r\rightarrow 0}\cH^{d}(B(x,r)\cap \d\Omega)/r^{d}<\infty\}.\]
The question whether NTA-ness of the domain is enough to obtain $\omega \ll \cH^{d}|_{\d\Omega}$ was already answered in the negative by Wolff in \cite{Wo}, with the impressive construction of the so-called Wolff snowflakes. Although, there was a question in \cite{Badger12} whether this could be true under the additional assumption that $\cH^{d}|_{\d\Omega}$ is locally finite. Very recently Azzam, Tolsa and the author \cite{AMT} demonstrated that there exists an NTA domain with very flat boundary for which $\cH^{d}|_{\d\Omega}$ is locally finite and yet, one can find a set $E \subset \d \Omega$ such that $\omega(E)>0=\cH^{d}(E)$.

However, it was left open whether one can show that $\cH^{d}|_{\d\Omega} \ll \omega$ relaxing the geometric conditions of the domain. It is this matter that we will deal with in the present paper. In fact, we show that $\cH^{d} \ll \omega$ on $\d\Omega$ up to a set of $\cH^{d}$-measure zero, under the assumption that the domain is uniform and its boundary is lower Ahlfors-David $d$-regular and $d$-rectifiable (all the definitions can be found in section \ref{sec:backgr}).

\begin{theorem}\label{th:main}
Suppose that $d\geq 2$. Let $\Omega \subset \bR^{d+1}$ be a uniform domain with lower Ahlfors David $d$-regular and $d$-rectifiable boundary $\d \Omega$. If $\cH^d|_{\partial \Omega}$ is locally finite, then $\cH^d|_{\partial \Omega} \ll \omega$, $\cH^d$-a.e. on $\d \Omega$.
\end{theorem}

The lower Ahlfors-David regularity may seem to be a technical condition but in fact, it is not. Indeed, on the one hand, the boundary of an NTA domain is always lower Ahlfors-David $d$-regular, while on the other, the conclusion of Theorem \ref{th:main} may fail once we relax the lower Ahlfors-David $d$-regularity assumption. We will show that Theorem \ref{th:main} is sharp in the following sense: 

For each $s \in (d-1,d)$, we can construct a uniform domain $\Omega \subset \bR^{d+1}$ such that
\begin{enumerate}
\item $\d \Omega$ is lower $s$-Ahlfors-David regular,
\item $\cH^d|_{\partial \Omega}$ is locally finite,
\end{enumerate}
and there exists a set $E \subset \partial \Omega$ for which $\cH^d(E)>0=\omega(E)$. 

An example of such a domain, constructed by J. Azzam, will be presented in the Appendix \ref{sec:Appendix}.

While putting the final touches to this manuscript , Jose  Mar\'ia Martell informed us that in a joint work with Akman, Badger and Hofmann in preparation, they have obtained independently Theorem \ref{th:main} under slightly stronger assumptions (in particular, assuming that $\d \Omega$ is Ahlfors-David $d$-regular).

\

\subsection*{Acknowledgements} We warmly thank J. Azzam for his encouragement and several discussions pertaining to this work and rectifiability, as well as for explaining the techniques developed in his earlier work on the same topic. The author was supported by the ERC grant 320501 of the European Research Council (FP7/2007-2013).

\section{Background material}\label{sec:backgr}
\begin{itemize}
\item If $A,B\subset \bR^{d+1}$, we let 
\[\dist(A,B)=\inf\{|x-y|:x\in A,y\in B\}, \;\; \dist(x,A)=\dist(\{x\},A),\]
\item $B(x,r)$ stands for the open ball of radius $r$ which is centered at $x$. We also denote by $\lambda B(x,r)=B(x,\lambda r)$.
\item  We will write $p \lesssim q$ if there is $C>0$ so that $p \leq C q$ and $p \lesssim_{M} q$ if the constant $C$ depends on the parameter $M$. We write $p \sim q$ to mean $p\lesssim q \lesssim p$ and define $p\sim_{M} q$ similarly. 
\item For $A\subset \bR^{d+1}$ and $s \in (0, d+1]$ we set
\[\cH^{s}_{\delta}(A)=\inf\left\{\sum r_{i}^{s}: A\subset \bigcup B(x_{i},r_{i}),x_{i}\in\bR^{d+1}\right\}.\]
Define the {\it $s$-dimensional Hausdorff measure} as
\[\cH^{s}(A)=\lim_{\delta\downarrow 0}\cH^{s}_{\delta}(A)\]
and the {\it $s$-dimensional Hausdorff content} as $\cH^{s}_{\infty}(A)$. See Chapter 4 of \cite{Mattila} for more details. \\
\end{itemize}

\begin{definition}
We say that a set $E \subset \bR^{d+1}$ is {\it Ahlfors-David $s$-regular} ({\it $s$-ADR}) if there is $C\geq 1$ so that
\begin{equation}
r^{s}/C\leq  \cH^{s}(B(x,r))\leq Cr^{s}\mbox{ for all }x\in E, 0<r<\diam E.
\label{e:regular}
\end{equation}

If a set $E \subset \bR^{d+1}$ satisfies only the lower (resp. upper) bound we shall call it {\it lower} (resp. {\it upper}) {\it Ahlfors-David $s$-regular}. 

\end{definition}
\begin{definition}
A domain $\Omega$ is called {\it uniform} if there is $C_\Omega>0$ so that for every $x,y\in \cnj{\Omega}$ there is a path $\gamma\subset\Omega$ connecting $x$ and $y$ such that
\begin{enumerate}[(a)]
\item if $\ell(\gamma)$ is the length of $\gamma$, then $ \ell(\gamma) \leq C_\Omega|x-y|$ and
\item for $z \in \gamma$, $\dist(z,\d\Omega)\geq \dist(z,\{x,y\})/C_\Omega$. 
\end{enumerate}
We call any such path a {\it good curve} connecting $x$ and $y$.
\end{definition}

\begin{definition}
We say that $\Omega$ satisfies the {\it interior Corkscrew condition} if for all $x \in \d \Omega$ and $r\in(0, \diam\d\Omega)$ there is a ball $B(z,r/C)\subseteq \Omega\cap B(x,r)$. We say that $\Omega$ satisfies the {\it exterior Corkscrew condition} if there is a ball $B(y,r/C)\subseteq B(x,r)\backslash \overline \Omega$ for all $x\in \d\Omega$ and $r\in(0,\diam \d\Omega)$. 
\label{d:cork}
\end{definition}

\begin{definition}
A domain $\Omega$ is called {\it non-tangentially accessible (NTA)} \cite{JK} if it is uniform and satisfies the exterior Corkscrew condition.
\label{d:NTA}
\end{definition}

We introduce  the notion of ``dyadic cubes" for a metric space (we may refer to them as ``metric cubes"). We will use the construction of Hyt\"onen and Martikainen from \cite{HytMart12}, which refines the originals of Christ \cite{Christ-T(b)} and David \cite{David88}.

\begin{theorem}\label{th:dyadic}
For $c_{0}<1/1000$, the following holds. Let $c_{1}=1/500$ and $\Sigma$ be a metric space. For each $n\in\bZ$ there is a collection $\cD_{n}$ of ``cubes,'' which are Borel subsets of $\Sigma$ such that
\begin{enumerate}
\item $\Sigma=\bigcup_{\Delta\in \cD_{n}}\Delta$ for every $n$,
\item if $\Delta,\Delta'\in \cD=\bigcup \cD_{n}$ and $\Delta\cap\Delta'\neq\emptyset$, then $\Delta\subseteq \Delta'$ or $\Delta'\subseteq \Delta$,
\item for $\Delta\in \cD_{n}$, there is $x_{\Delta}\in X_{n}$ so that if $B_{\Delta}=B(x_{\Delta},5c_{0}^{n})$, then
\[c_{1}B_{\Delta}\subseteq \Delta\subseteq B_{\Delta}.\]
\end{enumerate}
\label{t:Christ}
\end{theorem}

For $\Delta\in \cD_{n}$, define $\ell(\Delta)=5c_{0}^{n}$, so that $B_{\Delta}=B(x_{\Delta},\ell(\Delta))$. 
\begin{remark}
For $\Delta\in \cD_{n}$ and $\Delta'\in \cD_{m}$, we have $\ell(\Delta)/\ell(\Delta')=c_{0}^{n-m}$. 
\end{remark}
\begin{remark}
If $\Sigma$ is an ADR set then we may take $c_{0}=1/2$.
\end{remark}
We recall now the notion of {\it rectifiability} and its quantitative analogue ({\it uniform rectifiability}).

\begin{definition}\label{def:rec-lip}
If $E\subseteq \bR^{d+1}$ is a Borel set, we say that $E$ is {\it $n$-rectifiable} if $\cH^n(E\backslash \bigcup_{i=1}^{\infty} \Gamma_{i})=0$ where $\Gamma_{i}=f_{i}(E_{i})$, $E_{i}\subseteq \bR^{n}$, and $f_{i}:E_{i}\rightarrow \bR^{d+1}$ is Lipschitz. 
\end{definition}

One can find several characterizations of uniformly rectifiable sets in \cite{DS} and \cite{of-and-on}. We state here only two of them.

\begin{definition}\label{d:ur}
An Ahlfors-David $n$-regular set $\G \subseteq \bR^{d+1}$ is called {\it uniformly $n$-rectifiable} if there are constants $L>0$ and $c \in (0,1)$ such that, for all $x\in \G$ and $r\in(0,\diam \G)$, there exist $E\subseteq B(x,r)\cap \G$ with $\cH^{n}(E)\geq cr^{n}$ and an $L$-bi-Lipschitz embedding $g:E\rightarrow \bR^{n}$.
\end{definition}

\begin{remark}
If $\G$ is a bi-Lipschitz image of $\bR^{n}$, then it is uniformly $n$-rectifiable. 
\end{remark}

The characterisation that will be most convenient for us is the one given in terms of {\it bilateral $\beta$-numbers}: 

\[
\bbeta_{\G}(\Delta, P)={\sup_{x\in MB_\D \cap \G}\dist(x,P)/\ell(\D) + \sup_{y\in MB_\D\cap P} \dist(y, \G)/\ell(\D)},\]
where $P$ is an $n$-plane and $MB_\D$ stands for the ball $B(x_\D, M\ell(\D))$.

\begin{remark}
By the local compactness of the Grassmanian and the continuity of $\bbeta(\D,P)$ in $P$, there exists $P$ that infimizes $\bbeta(\D, P)$, and we define
\[\bbeta_{\G}(\D)=\inf_P \bbeta_{\G}(\D, P).\]
\end{remark}

\begin{theorem} \cite[Theorem 2.4]{of-and-on}
Let $\G$ be an Ahlfors-David $n$-regular set in $\bR^{d+1}$. Then $\G$ is uniformly rectifiable if and only if for every pair of positive constants $\ve\ll1$ and $M \gg 1$, there is a disjoint decomposition $\cD(\G) = \mathcal G \cup \cB$, such that the cubes in $\cB$ satisfy the a Carleson packing condition
\begin{equation}\label{eq:bwgl-packing}
\sum_{\D' \subset \D: \D' \in \cB} \cH^n(\D') \lesssim_{\ve, M }\cH^n(\D),\,\,\text{for all}\,\, \D \in \cD(\G),
\end{equation}
and such that for every $\D \in \mathcal G$, we have that $\bbeta_{\G}(\Delta)<\ve$.
\label{th:BWGL}
\end{theorem}

Finally we recall a useful corollary from \cite{Az}.

\begin{lemma}\cite[Corollary 3.4]{Az}
Let $\mu$ be a Borel measure, $\Sigma=\supp \mu$ and $E\subseteq \Delta_{0}\in \cD(\Sigma)$ be a Borel set. Let also $0<\delta<1<M<M_{0}/2$ and set
\[\cP_{M,\delta}=\{\Delta: \Delta\cap E\neq\emptyset, \exists \; x\in MB_{\Delta} \mbox{ such that }\dist(x,E)\geq \delta\ell(\Delta)\}.\]
Then there is $C_{1}>0$ so that, for all $\Delta'\subseteq \Delta_{0}$ in $\cD(\Sigma)$,
\begin{equation}
\sum_{\Delta\subseteq \Delta' \atop \Delta\in \cP_{M,\delta}} \mu(\Delta)\leq C_{1} \mu(\Delta').
\label{e:sumP}
\end{equation}
\label{c:porous}
\end{lemma}

\section{Whitney cubes and interior sub-domains}

For $n\in\bZ$, a {\it $(d+1)$-dimensional dyadic cube $Q$ of side length $2^{n}$} in $\bR^{d+1}$ is a $(d+1)$-fold Cartesian product of closed intervals of the form $[i2^{n},(i+1)2^{n}]$, where $i\in\bZ$.  We will denote by $\ell(Q)=2^{n}$ the side-length of $Q$ and by $\lambda Q$ the cube of the same center as $Q$ and edges parallel to the coordinate axes but side-length $\lambda\ell(Q)$.

\begin{definition}[Whitney Cubes]
For an open set $\Omega\subseteq\bR^{d+1}$ we will denote by $\cW(\Omega)$ the set of maximal dyadic cubes $Q\subseteq \Omega$ such that $3 Q\cap \Omega^{c}=\emptyset$. These cubes have disjoint interiors and satisfy the following properties:
\begin{enumerate}
\item $\ell(Q)\leq \dist(x,\Omega^{c})\leq 4\diam Q$ for all $x\in Q$,
\item $(1-\sqrt{d+1}\frac{\lambda-1}{2})\ell(Q)\leq \dist(x,\Omega^{c})\leq  (4+(\lambda-1)/2)\diam Q$ for all $x\in \lambda Q$ if $\lambda\geq 1$ is close enough to $1$ (depending on $d$)
\item If $Q,R\in \cW(\Omega)$ and $Q \cap R \neq \emptyset$, then $\ell(Q)\sim_{d}\ell(R)$.
\item $\sum_{Q\in \cW(\Omega}\one_{2\lambda Q}\lec_{d}\one_{\Omega}$ for sufficiently small  $\lambda> 1$ (depending on $d$).
\end{enumerate}
\label{d:Whitney}
\end{definition}

\begin{itemize}
\item We say that $Q,R\in \cW(\Omega)$ are {\it adjacent} if $Q\cap R\neq\emptyset$ and we write $Q\sim R$. 
\item We denote by $P_{Q,R}$ the shortest path $Q=Q_{0},...,Q_{k}=R$ of Whitney cubes such that $ Q_{j}\sim Q_{j+1}$ for $j=0,...,k-1$ and define the distance $d_{\Omega}(Q,R)=k+1$. 
\end{itemize}

We can now state an equivalent characterization of uniformity .

\begin{theorem}[Alternate characterization of uniform domains] A domain $\Omega$ is uniform if and only if it satisfies the interior Corkscrew condition and there is $N_{\Omega}:[0,\infty)\to [0,\infty)$ increasing such that,
\begin{equation}
d_{\Omega}(Q,R)\leq N(\dist(Q,R)/\min\{\ell(Q),\ell(R)\}) \mbox{ for all }Q,R\in \cW(\Omega).
\label{e:d<N}
\end{equation}
\label{t:AHMNT}
\end{theorem}

We state here a method to construct a uniform sub-domain $\Omega^-$ around a prescribed portion of a uniform domain $\Omega$. This construction is pretty standard but a proof can be found for example in \cite[Lemma 4.1]{Az}.

\begin{lemma}
Let $\Omega\subseteq\bR^{d+1}$ be a uniform domain and let $E\subseteq B(x_{0},r_{0})\cap \d\Omega$ be compact where $x_{0}\in \d\Omega$ and $r_{0}\in(0,\diam\d\Omega)$. Set $C_{0}>0$ and 
\[\cC_{E}^{-}=\{Q\in \cW(\Omega): C_{0}Q\cap E\neq\emptyset, \ell(Q)\leq  r_{0}\}.\]
For some constant $\tilde{C}>0$, set
\[\widetilde{\cC_{E}}^{-}=\{Q:Q\in P_{Q_{1},Q_{2}}\mbox{ for some }Q_{1},Q_{2}\in \cC_{E}^{-} \mbox{ with } d_{\Omega}(Q_{1},Q_{2})\leq \tilde{C}\}.\]
For $\lambda>1$, set
\[\Omega_{E}^{-}=\ps{\bigcup_{Q\in \widetilde{\cC_{E}}^{-}}\lambda Q}^{\circ}.\]
Then for $C_{0}$ and $\tilde{C}$ large enough and $\lambda>1$ close enough to $1$ (each depending only on $C_\Omega$ and $d$), $\Omega_{E}^{-}$ is a uniform domain contained in $B(x_{0},C^{-}r_{0})$ and $\diam \d\Omega_{E}^{-}\geq r_{0}/C^{-}$, for some $C^{-}:=C_{\Omega_{E}^{-}}=C_{\Omega_{E}^{-}}(d,C_{0},\lambda,C_\Omega)$. Moreover, $\d\Omega_{E}^{-}\cap \d\Omega= \cnj{E}$.
\label{l:omegaE-}
\end{lemma}

\begin{remark}\label{rem:harnack}
Let
$$\d{\cC_{E}^{-}}=\{Q\in \cC_{E}^{-}: Q\sim Q'\mbox{ for some } Q'\in \cW(\Omega)\backslash \cC_{E}^{-}\} $$
and 
$$\d\widetilde{\cC}_{E}^{-}=\{Q\in \widetilde{\cC_{E}}^{-}: Q\sim Q'\mbox{ for some } Q'\in \cW(\Omega)\backslash \widetilde{\cC_{E}}^{-}\}. $$
For each $R \in \d\widetilde\cC_{E}^{-}$ there exist at most $N=N(\tilde C, d)$ cubes $Q \in \d{\cC_{E}^{-}}$ with $Q'=R$.
\end{remark}

\section{Main lemmas}
Another characterization of rectifiability, which will be suitable for our purpose, is described in the following proposition.

\begin{proposition}\label{prop:rec-bilip}
$E\subseteq \bR^{d+1}$ is a {\it $n$-rectifiable} set if and only if  $\cH^n(E\backslash \bigcup_{i=1}^{\infty} \Gamma_{i})=0$ where $\Gamma_{i}=F_{i}(\bR^n)$ and $F_{i}:\bR^n\rightarrow \bR^{d+1}$ is bi-Lipschitz. 
\end{proposition}

For the proof we need the following theorem.
\begin{theorem} \cite[Theorem II]{hardsard}.
Let $D\geq d\geq 1$ and  $0<\kappa<1$ be given. 
There are constants $C'=C'(d)>0$ and $M=M(\kappa,d)$ such that if
$f:\bR^{d}\rightarrow\bR^{D}$ is a $1$-Lipschitz function, then there are sets $\Sigma_{1},...,\Sigma_{M}$ such that 
\begin{equation}
\cH_{\infty}^{d}\ps{f\ps{[0,1]^{d}\backslash\bigcup_{i=1}^{M} \Sigma_{i}}}\leq C' \kappa
\label{e:inverse-theorem-1}
\end{equation}
and such that if $\Sigma_{i}\neq\emptyset$, there is $F_{i}:\bR^{d}\rightarrow\bR^D$ which is $L_{0}$-bi-Lipschitz, $L_{0}\sim_{D}\kappa^{-1}$, so that 
\begin{equation}
F_{i}|_{\Sigma_i}= f|_{\Sigma_{i}}.
\label{e:extends}
\end{equation}
\label{t:hard-sard}
\end{theorem}

\begin{proof}[Proof of Proposition \ref{prop:rec-bilip}] 
The sufficiency part is straightforward. For the necessity part, we let $E \subset \bR^{d+1}$ be a $n$-rectifiable set. Then, by definition \ref{def:rec-lip}, there exist $\Gamma_{i}=f_{i}(E_{i})$, where $E_{i}\subseteq \bR^{n}$ and $f_{i}:E_{i}\rightarrow \bR^{d+1}$ Lipschitz, such that $\cH^n(E\backslash \bigcup_{i=1}^{\infty} \Gamma_{i})=0$. We extend $f_i$ to Lipschitz functions$\tilde f_i: \bR^n \rightarrow \bR^{d+1}$ and then we cover $ \bR^n$ by $n$-dimensional cubes $\{Q_j\}_{j=1}^\infty$ of unit length. 

Fix such a cube $Q_j$ and then fix a Lipschitz extension $\tilde f_i$ restricted to $Q_j$. If $k \in \bN$ and $\delta=1/k$, by Theorem \ref{t:hard-sard}, we find $M=M(k, n)$ sets $\Sigma^{i,j}_{1},...,\Sigma^{i,j}_{M}$ such that $\cH_{\infty}^{n}\ps{\tilde f_i \ps{Q_j\backslash\bigcup_{\ell=1}^{M} \Sigma^{i,j}_{\ell}}}\leq C' k^{-1}$. Additionally, there are $F^{i,j}_{\ell}:\bR^{n}\rightarrow\bR^{d+1}$ which are $L_{0}$-bi-Lipschitz, with $L_{0}\sim_{d} k$, so that $F^{i,j}_{\ell}|_{\Sigma^{i,j}_\ell}= \tilde f_i|_{\Sigma^{i,j}_{\ell}}$. 

If we apply this to each $\tilde f_i$ and each cube $Q_j$, it is easy to see that $\{F^{i,j}_{\ell}\}_{i,j,\ell}$ is our collection of bi-Lipschitz maps. 
%, we pick up a collection of at most countably many sets $\{\Sigma_\ell\}_\ell$ so that $\tilde f_i (\cup_\ell \Sigma_\ell)$ exhausts up to a set of $\cH^d-$measure zero $\tilde f_i (\bR^d)$. Trivially, it exhausts $f_i (E_i)$ as well and the proof is now complete.

\end{proof}

\begin{lemma}\label{lem:BP-E'}
 %so that $H^d|_\G$ is Radon.
Let $\G  \subset \bR^{d+1}$ be a closed set. Suppose that $\Delta_0 \in \cD(\G)$ and a Borel set $E \subset \Delta_0$ so that $0<\cH^d(E)<\infty$. Then there exist $C_0>1$ and a Borel set $E' \subset E$ such that
\begin{enumerate}
\item  $\cH^d(E') \geq \frac{1}{2}\cH^d(E)$,
\item $\cH^d(E \cap \Delta) \geq C_0^{-1} \cH^d(\Delta)$, for every $\Delta \in \cD(\G)$ for which $\Delta \subset \Delta_0$ and $\Delta \cap E' \neq \emptyset$.
\end{enumerate}
\end{lemma}

\begin{proof}
Let $\{\D_i\}_{i \in I}$ be the maximal sub-collection of metric cubes in $\cD(\G)$ such that $\D_i \cap E \neq \emptyset$, $\D_i \subset \D_0$ and 
$$\cH^d(\D_i \cap E)\leq \delta \cH^d(\D_i),$$
for some $\delta>0$ to be chosen. Define $E':=E \backslash \bigcup_{i \in I} \D_i$ and note that
\begin{align*}
\cH^d(E') =&\cH^d(E) - \sum_{i \in I} \cH^d(\D_i \cap E) \\
\geq &\cH^d(E) - \delta \sum_{i \in I} \cH^d(\D_i) \\
\geq &\cH^d(E) \Big(1- \delta \frac{\cH^d(\D_0)}{\cH^d(E)}\Big).
\end{align*}
We conclude by choosing $\delta =\cH^d(E)/2\cH^d(\D_0)$.
\end{proof}

\begin{lemma}\label{lem:BP-Bad}
Let $\G \subset \bR^{d+1}$ be an Ahlfors-David $d$-regular closed set, $\Delta_0 \in \cD(\G)$ and $E$ be a Borel subset of $\Delta_0$ so that $0<\cH^d(E)<\infty$. Suppose that $E'$ is the subset of $E$ obtained by Lemma \ref{lem:BP-E'} and $\cB \subset \cD(\G)$ is a sub-collection of metric cubes such that for each $\Delta \in \cD(\G)$ we have that  
$$\sum_{\Delta'  \in \cB: \Delta' \subset \Delta} \cH^d(\Delta') \lesssim \cH^d(\Delta).$$
Then for every $\Delta \subset \Delta_0$ for which $\Delta \cap E' \neq \emptyset$, there exists $\Delta' \subset \Delta$ such that $\Delta' \in \cD(\G) \backslash \cB$ and $\ell(\Delta') \sim \ell(\Delta)$.
\end{lemma}

\begin{proof}
We let $\D \in \cD$ such that $\D \cap E' \neq \emptyset$ and $\Delta \subset \D_0$. Define now
$$ \cM_k=\{ \D' \in \cD: \D' \subset \D, \ell(\D)/\ell(\D') \sim 2^{k} \,\, \text{and}\,\,\D' \cap E \neq \emptyset \}.$$
By Lemma \ref{lem:BP-E'}, we obtain that
\begin{align*}
\cH^d(\D) &\lesssim \cH^d(\D \cap E) \leq \sum_{\D' \in \cM_k} \cH^d(\D' \cap E)\lesssim2^{-kd} \cH^d(\D)  |\cM_k|,
\end{align*}
where $|\cM_k|$ stands for the cardinality of $\cM_k$. Therefore, $|\cM_k| \gec 2^{kd}$.

Take now all the  metric cubes $\D' \in \bigcup_{k=1}^N \cM_k$ and notice that in the case that $\bigcup_{k=1}^N \cM_k \subset \cB$ we have that
\begin{align*}
N \cH^d(\D) & \lesssim \cH^d(\D) \sum_{k=1}^N  2^{-kd} |\cM_k| \sim \sum_{k=1}^N \sum_{\D' \in \cM_k} \cH^d(\D')  \\
&\leq \sum_{\Delta'  \in \cB: \Delta' \subset \Delta} \cH^d(\Delta') \lesssim \cH^d(\Delta).
\end{align*}
If we choose $N>0$ sufficiently large, we reach a contradiction and the lemma follows.
\end{proof}

\
\section{Core of the proof of Theorem \ref{th:main}}\label{sec:Core}
Let $\Omega \subset \bR^{d+1}$ be as in Theorem \ref{th:main}. Since $\partial \Omega$ is $d$-rectifiable we can apply Proposition \ref{prop:rec-bilip} and find a countable union of bi-Lipschitz images that exhausts $\d \Omega$ up to a set of $\cH^d$-measure zero. We fix such an image $F_i(\bR^d)$ and denote it by $\Gamma$. Let $F:= \partial \Omega \cap \Gamma$. Then by Lebesgue's density theorem, for $\cH^d$-a.e. $x \in F$, it holds that

$$\lim_{r \to 0} \frac{\cH^d(B(x,r) \cap F)}{\cH^d( B(x,r) \cap \partial\Omega)} \to 1.$$
Therefore, for $\cH^d$-a.e. $x \in F$, there exists $r_x>0$ such that for every $0<r<r_x$, $\cH^d(B(x,r) \cap F) \geq \cH^d( B(x,r) \cap \partial\Omega)/2>0$. 

Fix now $x_0 \in F$ and $r_0<r_{x_0}$.%, where $c_0>0$ is chosen so that we can find $\Delta_0 \in \D(\G)$ that contains $B(x_0,r_0) \subset \D_0$ for which $r_{x_0} \sim \ell(\D_0) \leq r_{x_0}$.

\begin{lemma}\label{lem:dense-flat}
Let $\Omega \subset \bR^{d+1}$ be a uniform domain and $E$ be a compact subset of $ B(x_0,r_0) \cap \partial \Omega \cap \Gamma$ such that $\cH^d(E)>0$. Let also $M=2C_\Omega+1$ and $\varepsilon>0$ be sufficiently small. Suppose that $\Delta \in \cD(\G)$ has the following properties:
\begin{itemize}
\item[1)] (flatness) $\bbeta_\G(\Delta) <\varepsilon$,
\item[2)]  (density) For every $x \in M B_\Delta \cap \G$ and $\dist(x, E) \leq \varepsilon \ell(\Delta)$.
\end{itemize}
 Then there exists a ball $B_0 \subset B_\Delta \backslash \cnj \Omega$ such that $r(B_0) \sim \ell(\D)$.
\end{lemma}

\begin{proof} 
Let $P$ be the hyperplane that infimizes $\bbeta_\G(\Delta)$ and $P'$ the hyperplane parallel to $P$ passing through $x_\D$ (the center of $B_\D$). Then $\bbeta_\G(\D, P') \leq 2 \ve$. Without loss of generality we assume $x_\D=0$ and $P'=\bR^d$. 

Let $\widetilde B$ be a Corkscrew ball in $\Omega$ for $B_\D$ with radius $r(\widetilde B) \sim \ell(\D)$. We claim that every $x \in 1/2\widetilde B$ satisfies $\dist(x, \bR^d) \gec \ell(\D)$. Indeed, if this was not the case, we would have that $\dist(1/2\widetilde B, \bR^d) \ll \ell(\D)$ and therefore, by the density and flatness condition for $\D$, $\widetilde B \cap \Omega \neq \emptyset$. But this violates that $\widetilde B$ is a Corkscrew ball in $\Omega$ and proves our claim.

Fix $x \in 1/2\widetilde B$ and let $y$ be in the reflection of $1/2\widetilde B$ across $\bR^d$. We will show that $y$ cannot lie in $\Omega$. We assume to the contrary that both $x$ and $y$ are in $\Omega$. Then, by the uniformity of $\Omega$, there exists a good curve $\gamma$ connecting $x$ and $y$ (notice that by the choice of $M$ it is always true that $\gamma \subset M B_\Delta$). Therefore, there exists  $z \in \bR^d \cap\,\gamma \cap \, M B_\Delta$. If $ z_\G \in \G$ is the point that realizes the distance $\dist(z, \Gamma)$, we have that $|z-z_\G| \leq 2 \ve \ell(\D)$, using that $\bbeta_\G(\D, \bR^d) < 2\ve$. This, in turn, by the density of $MB_\D \cap \G$ in $E$, implies that $d(z, E) \leq 3 \ve  \ell(\D)$. Using the ``goodness" of the curve $\gamma$ we obtain that 
\begin{equation} \label{eq:good-fd}
\dist(z,\{x,y\})/c\leq \dist(z, \partial \Omega) \leq \dist(z, E) \leq 3 \ve \ell(\D).
\end{equation}
But since $\dist(x, \bR^d) \sim \dist(y, \bR^d) \sim \ell(\D)$ and $\ve$ is sufficiently small, we reach a contradiction and this concludes the theorem.

%Choose now $\tilde y \in B_\D$ such that  $|\tilde y-z| \geq 3 \ve c\ell(\D)$ and $\dist(\tilde y, \partial \Omega \cap B_\D) \sim \ell(\D)$. Then necessarily  $\tilde y \in B_\D \backslash \cnj \Omega$ and thus, it is an exterior Corkscrew point for $B_\Delta$.

\end{proof}

Since $\cH^d|_{\d\Omega}$ is Radon, we can always find $E\subset \partial \Omega \cap \Gamma \cap B(x_0,r_0)$ compact with $\cH^d(E)>0$. Let now $E' \subset E$ be the set obtained from Lemma \ref{lem:BP-E'} and construct a uniform domain $\Omega_{E'}^-$ around $E'$ as in Lemma \ref{l:omegaE-}. We will show that $\Omega_{E'}^-$  is an NTA domain. 

\begin{lemma}
$\Omega_{E'}^{-}$ satisfies the exterior Corkscrew condition. 
\end{lemma}

\begin{proof}
It is enough to show that for every $x\in \partial \Omega_{E'}^-$ there exists a ball $B_0 \subset B(x,r)\setminus \cnj{\Omega^{-}_{E'}}$ with radius $r(B_0) \sim r$. We call such $B_0$ a Corkscrew ball.

Let $\dist(x,E')<r/2$ and $x'\in E$ be so that $|x'-x|<r/2$. Then there is $\D \in \cD(\G)$ containing $x'$ with $\ell(\D) \sim r$ such that \[B_\D\subseteq  B(x',r/2)\subseteq B(x,r).\] If $\D$ satisfies the flatness and density conditions of Lemma \ref{lem:dense-flat}, then the existence of a ball $B_0$ with the desired properties follows by that lemma. If not, we set $\cB$ to be collection of cubes for which either $\bbeta_\G(\Delta) \geq\varepsilon$ or there exists $x\in M B_\Delta$ such that $\dist(x, E') > \varepsilon \ell(\Delta)$. In light of Theorem \ref{th:BWGL} and Lemma \ref{c:porous}, this is a Carleson family and thus, by Lemma \ref{lem:BP-Bad} there exists $\D' \subset \D$ such that $\D' \in \cD(\G) \backslash \cB$ and $\ell(\D')\sim \ell(\D) \sim r$. We apply once more Lemma \ref{lem:dense-flat} and obtain a Corkscrew ball $B_0$.

Let $\dist(x,E')\geq r/2$. Then there exists $Q\in \d \widetilde{\cC}_{E'}^{-}$ such that $x\in \d \lambda Q$. If $R\in \cW(\Omega)$ is the Whitney cube containing $x$, it is clear that $R\not\in \widetilde{\cC}_{E'}^{-}$. Since $\ell(R) \sim\ell(Q')$ for any Whitney cube $Q' \sim R$, we have that $R'=R\backslash\bigcup_{Q'\in\widetilde{\cC}_{E'}^{-}} \lambda Q'$ is a rectangular prism with all side-lengths comparable to $\ell(R)\sim_{d} \ell(Q)$. In light of $C_{0}Q\cap E'\neq\emptyset$ and $x\in \lambda Q\subseteq C_{0}Q$, it holds that
\[r\leq 2\dist(x,E') \leq 2\diam  C_{0}Q\lec_{d} \ell(R),\] 
and clearly $B(x, r) \cap R'$ contains a Corkscrew ball of radius $\sim r$.

\end{proof}

It only remains to show that the boundary of the new domain $\Omega^{-}_{E'}$ has finite $d$-Hausdorff measure. 

\begin{lemma}
$\cH^d( \d \Omega^{-}_{E'}) < \infty$.
\end{lemma}

\begin{proof}

If $Q \in \d \cC_{E'}$, there exists $Q' \sim Q$ which is not in $\d \cC^-_{E'}$, i.e., $C_0 Q' \cap  E'= \emptyset$. We can pick $C_0>0$ so large  that there exists $\D \in \cD(\partial \Omega)$ which is contained in $C_0Q'$ and $\ell(\D) \sim \ell(Q') \sim \ell(Q)$. Let $\Delta_Q \in \cD(\partial \Omega)$ be the maximal metric cube such that $\Delta_Q \in C_0 Q'$, $\ell(\D) \sim \ell(Q') \sim \ell(Q)$ and $3 B_{\Delta_Q} \cap E' = \emptyset$. We also let $y_Q=x_{\D_Q}$ (recall that $x_{\D_Q}$ is the center of $B_{\D_Q}$). 

\Claim For any fixed metric cube $\D \in \cD(\d \Omega)$, there exists $N_{0}=N_{0}(d)>0$ so that $\sharp \{ Q\in \d \cC^-_{E'}: \dq = \Delta\} \leq N_0$.  To see this, fix $\Delta \in \cD(\d \Omega)$ and suppose that $\Delta_{Q}=\Delta$, for some $Q \in \d \cC^-_{E'}$. By the definition of $\D_Q$, there exists some (possibly large) positive absolute constant $\sigma$ so that any cube $Q\in \d \cC^-_{E'}$ for which $\Delta=\dq$ is contained in the ball $B(x_\D, \sigma \ell(\Delta))$. Since all $Q \in \d \cC^-_{E'}$ such that $\dq=\Delta$ are disjoint and have comparable side-lengths, by volume considerations the claim follows.

Notice now that for $Q \in \d \cC^-_{E'}$ we have $\ell(Q) \sim \dist(Q, E') \lesssim r_0$ and thus, $\ell(\dq) \lesssim r_0$. Moreover, $\dist(\dq, x_0) \leq \dist(\dq, Q) + \dist(Q, x_0) \leq \dist(y_Q, Q) + r_0 \lesssim \ell(Q) + r_0 \lesssim r_0$. 

We set $\mathcal{S}:=\{ \Delta \in \cD: \Delta=\dq\,\, \textup{for some} \,\, Q \in \d \cC^-_{E'}\}$ which is a disjoint family of cubes. Note also that there exists $A>0$ so that $\mathcal S$ is contained in $B(x_0, A r_0)$. This follows easily from $\ell(\dq) \lesssim r_0$ and $\dist(\dq, x_0)\lesssim r_0$. Therefore, using the lower $d$-ADR property of $\partial \Omega$ we obtain that

\begin{align*}
\sum_{Q\in \d\cC^-_{E'}}\ell(Q)^{d}  &\sim \sum_{Q\in \d\cC^-_{E'}}  \ell(\dq)^d  \lesssim \sum_{Q\in \d \cC^-_{E'}}  \cH^d(\D_Q \cap \d\Omega)\\
&\lesssim_{N_0} \sum_{\D \in S} \cH^d(\D \cap \d\Omega)
 \leq  \cH^d(B(x_0, A r_0) \cap \d\Omega),
\end{align*}
where in the penultimate inequality we used that there are at most $N_0$ number of metric cubes such that $\Delta=\D_Q$ and in the last one that $\mathcal S$ is contained in $B(x_0, A r_0)$. Since $E'\subset B(x_0, r_0) \cap \d \Omega$ and $\cH^d|_{\partial \Omega}$ is a locally finite measure, the lemma follows from Remark \ref{rem:harnack} and the definition of (the boundary of) $\Omega^-_{E'}$.
\end{proof}

Let us denote by $\omega^-$  and $\omega$ the harmonic measure in the domain $\Omega^-_{E'}$ and $\Omega$ respectively, with pole at a fixed point of $\Omega^-_{E'}$ (and thus, of $\Omega$) so that its distance to the boundary of $\Omega^-_{E'}$ is comparable to $r_0$. Then, by \cite{Badger12} we conclude that $\cH^d|_{\partial \Omega^-_{E'}} \ll \omega^-$ and by the maximum principle, this implies that $\cH^d|_{E'} \ll \omega|_{E'}$.

\section{End of the proof of Theorem \ref{th:main}}

Suppose that there exists $F\subset \d \Omega$ such that $\omega(F)=0$ but $\cH^d(F)>0$. 	Then there exists a bi-Lipschitz image $\G$ such that $\cH^d(\G \cap F)>0$. Arguing as in the beginning of section \ref{sec:Core}, we pick $x_0 \in \Gamma \cap F$ and $r_0>0$ such that 
$$\cH^d(\G \cap F \cap B(x_0,r_0)) \gec \cH^d(\d \Omega\cap B(x_0,r_0) )>0.$$
Moreover, since $\cH^d|_{\partial \Omega}$ is Radon, we can find a compact set $E \subset B(x,r) \cap \G \cap F$ such that 
$$\cH^d(E) \gec \cH^d(B(x,r) \cap \G \cap F)>0.$$
Let now $E' \subset E$ be as in Lemma \ref{lem:BP-E'} and recall that $\cH^d(E')>0$. The latter implies that $\omega(E')>0$ since $\cH^d|_{E'}\ll \omega|_{E'}$. Then
$$0<\omega(E') \leq  \omega(E) \leq \omega(F) =0,$$
which leads us to a contradiction. Therefore, $\cH^d\ll \omega$ on $\d \Omega$ apart from a set of $\cH^d$-measure zero, which concludes the proof of Theorem \ref{th:main}.

\appendix

\section{} \label{sec:Appendix}

We present now the construction of the counterexample mentioned in section \ref{sec:intro}

Let $Q_0$ be the unit cube of $\bR^{d+1}$, $s \in (d-1,d)$ and $E\subset  Q_0$ is an Ahlfors-David $s$-regular set so that its complement is a uniform domain. Let $E_{2^{-n}}$ denote the union of all dyadic cubes of side-length $2^{-n}$ that intersect $E$. Then 
\begin{align}\label{eq:counter}
 \cH^d(\d E_{2^{-n}}) &\lesssim \sum_{Q\cap E \neq \emptyset: \ell(Q)=2^{-n}} \ell(Q)^d \notag \\
 &=2^{-n(d-s)}\sum_{Q\cap E \neq \emptyset: \ell(Q)=2^{-n}} \ell(Q)^s \lesssim 2^{-n(d-s)} \cH^s(E).
 \end{align}
Let  $\cW(\bR^{d+1})$ be for the Whitney decomposition of the upper half-space $\bR^{d+1}_+$. For each $W \in \cW(\bR^{d+1}_+)$, we let $T_W$ be the affine similarity that maps $Q_0$ to $W$ and set $E_W=T_W(E_{\ell(W)})$ so that
$$\cH^d(\d E_W)=\ell(W)^d \cH^d(\d E_{\ell(W)}) \lesssim \ell(W)^{2d-s} \cH^s(E),$$
where in the last inequality we used \eqref{eq:counter}. This estimate implies that if we define $\Omega:= \bR^{d+1}_+ \backslash \bigcup_{W \in \cW(\bR^{d+1}_+)} E_W$ then $\cH^d|_{\d \Omega}$ is locally finite. By construction it is not hard to see that $\Omega$ is uniform and its boundary $d$-rectifiable.

Notice now that by the Ahlfors-David $s$-regularity of $E$ one can deduce that $\cH^s_\infty(B \cap \d \Omega) \gec r(B)^s$ (with uniform contants), where $B$ is a ball of radius $r(B)$ centered on $\d\Omega$ and $\cH^s_\infty$ stands for the $s$-Hausdorff content. Therefore, by a result proved by Bourgain in \cite{Bo} (for a proof see also \cite[Lemma 4.1]{AMT2}) we have that there exists $c_0 \in (0,1)$ such that $\omega^{x_B}(B)>c_0$, where $B$ is a ball centered on $\d \Omega$ and $\omega^{x_B}$ is the harmonic measure in $\Omega$ with pole at $x_B$ (a Corkscrew point of $B$). With this in hand, we combine \cite[Lemma 4.2]{AMT2} and \cite[Lemma 3.6]{Ai} and obtain that the harmonic measure in $\Omega$ is doubling. 

Take now a ball $B(x,r)$ such that $ x \in \bR^d$. Note that there exists a Whitney cube $W_0 \subset B(x,r) \cap \bR^{d+1}_+$ such that $\ell(W_0) \sim r$. Then, by doubling, $ \omega(B(x,r)) \lesssim \omega(W_0)$. By Lebesgue's density theorem we have that,
$$ 1 \lec \lim_{r \to 0} \frac{\omega(W_0)}{\omega(B(x,r))} \leq \lim_{r \to 0} \frac{\omega(B(x,r) \setminus \bR^d)}{\omega(B(x,r))} \to 0, \,\,\text{for} \,\,\omega\text{--a.e.}\,\, x \in \bR^d.$$
Therefore, $\omega(\bR^d)=0$ and for any set $F \subset \bR^d$ such that $0<\cH^d(F)<\infty$ we will have $\omega (F)=0$, which concludes our proof.
%\begin{lemma}\label{lem:reduction}
%Let $F \subset \bR^{d+1}$ be a Borel  set and suppose that for $\cH^d-$a.e. $x\in F $ there exists $r_x>0$ so that the following holds: for every $r \in (0, r_x)$ we can find $F_{x,r} \subset F \cap B(x,r)$ such that $\cH^d (F_{x,r}) \geq C^{-1} \cH^d(B(x,r) \cap F)$ for some constant $C>1$. Then there holds that 
%\begin{equation}\label{eq:red-bigpiece}
%\cH^d\big(F \backslash \bigcup_{x \in F} \bigcup_{r \in (0, r_x)} F_{x,r}\big)=0
%\end{equation}
%\end{lemma}

\bibliographystyle{amsplain}

\def\cprime{$'$}
\providecommand{\bysame}{\leavevmode\hbox to3em{\hrulefill}\thinspace}
\providecommand{\MR}{\relax\ifhmode\unskip\space\fi MR }
\providecommand{\MRhref}[2]{%
  \href{http://www.ams.org/mathscinet-getitem?mr=#1}{#2}
}
\providecommand{\href}[2]{#2}

\end{document}